\pdfoutput=1
\documentclass[11pt]{article}

\usepackage[utf8]{inputenc}
\def\final{1}
\usepackage{color}
\usepackage[small,raggedright]{subfigure}
\usepackage{url}
\usepackage{graphicx}
\usepackage{amsmath}
\usepackage{amsthm}
\usepackage{bm}
\usepackage[plainpages=false]{hyperref}

\usepackage{titlesec}
\titleformat{\section}
{\normalfont\large\bfseries\filcenter}{\thesection.}{1 ex}{}
\titleformat{\subsection}[runin]
{\normalfont\normalsize\bfseries\filcenter}{\thesubsection.}{1 ex}{}

\usepackage[charter]{mathdesign}
\usepackage[mathcal]{eucal}

\usepackage[margin=1.2 in]{geometry}

\usepackage[square,numbers,sort&compress]{natbib}

\ifnum\final=0
\newcommand{\mynote}[1]{\marginpar{\tiny\sf #1}}
\else
\newcommand{\mynote}[1]{}
\fi

\usepackage{theomac}
\newtheoremWithMacro{theorem}{Theorem}
\newtheoremWithMacro{lemma}{Lemma}[section]
\newtheoremWithMacro{corollary}{Corollary}
\newtheoremWithMacro{prop}[lemma]{Proposition}

\newcommand{\figref}[1]{Figure \ref{fig:#1}}

\newcommand{\lemref}[1]{Lemma \ref{lemma:#1}}

\newcommand{\propref}[1]{Proposition \ref{prop:#1}}

\newcommand{\theoref}[1]{Theorem \ref{theo:#1}}
\newcommand{\secref}[1]{Section \ref{sec:#1}}
\renewcommand{\eqref}[1]{(\ref{eq:#1})}

\newcommand{\lemlab}[1]{\label{lemma:#1}}
\newcommand{\proplab}[1]{\label{prop:#1}}
\newcommand{\theolab}[1]{\label{theo:#1}}
\newcommand{\seclab}[1]{\label{sec:#1}}
\newcommand{\eqlab}[1]{\label{eq:#1}}

\renewcommand{\vec}[1]{\mathbf{#1}}

\newcommand{\iprod}[2]{\left\langle {#1},{#2}\right\rangle}

\newcommand{\bgamma}{\bm{\gamma}}

\newcommand{\Euc}{\operatorname{Euc}}

\newcommand{\eop}{\hfill$\qed$}

\newcommand{\R}{\mathbb{R}}
\newcommand{\Z}{\mathbb{Z}}

\DeclareMathOperator{\HH}{H}
\begin{document}
\title{Generic rigidity of reflection frameworks}
\author{Justin Malestein\thanks{Temple University, \url{justmale@temple.edu}}
\and Louis Theran\thanks{Institut für Mathematik,
Diskrete Geometrie, Freie Universität Berlin, \url{theran@math.fu-berlin.de}}}
\date{}
\maketitle
\begin{abstract}
\begin{normalsize}
We give a combinatorial characterization of generic minimally rigid
reflection frameworks.  The main new idea is to study a pair of direction networks on
the same graph such that one admits faithful realizations and the other has only
collapsed realizations.  In terms of infinitesimal rigidity, realizations of the
former produce a framework and the latter certifies that this framework
is infinitesimally rigid.
\end{normalsize}
\end{abstract}

\section{Introduction} \seclab{intro}
A \emph{reflection framework} is a planar structure made of \emph{fixed-length bars} connected by
\emph{universal joints} with full rotational freedom.  Additionally, the bars and joints  are symmetric with
respect to a reflection through a fixed axis.  The allowed motions preserve the \emph{length} and
\emph{connectivity} of the bars and \emph{symmetry} with respect to some reflection.
This model is very similar to that of \emph{cone frameworks} that we introduced in \cite{MT11}; the
difference is that the symmetry group $\Z/2\Z$ acts on the plane by reflection instead of rotation
through angle $\pi$.

When all the allowed motions are Euclidean isometries, a reflection framework is \emph{rigid} and otherwise it is
\emph{flexible}.  In this paper, we give a \emph{combinatorial} characterization of minimally rigid, generic
reflection frameworks.

\subsection{The algebraic setup and combinatorial model}
Formally a reflection framework is given by a triple $(\tilde{G},\varphi,\tilde{\bm{\ell}})$,
where $\tilde{G}$ is a finite graph, $\varphi$ is a $\Z/2\Z$-action on $\tilde{G}$ that is
free on the vertices and edges, and $\tilde{\bm{\ell}} = (\ell_{ij})_{ij\in E(\tilde{G})}$
is a vector of non-negative \emph{edge lengths} assigned to the edges of $\tilde{G}$.  A
\emph{realization} $\tilde{G}(\vec p,\Phi)$ is an assignment of points $\vec p = (\vec p_i)_{i\in V(\tilde{G})}$
and a representation of $\Z/2\Z$ by a reflection $\Phi\in \Euc(2)$ such that:
\begin{eqnarray}\eqlab{lengths-1}
||\vec p_j - \vec p_i||^2 = \ell_{ij}^2 & \qquad \text{for all edges $ij\in E(\tilde{G})$} \\
\eqlab{lengths-2}
\vec p_{\varphi(\gamma)\cdot i} = \Phi(\gamma)\cdot\vec p_i & \qquad
\text{for all $\gamma\in \Z/2\Z$ and $i\in V(\tilde{G})$}
\end{eqnarray}
The set of all realizations is defined to be the \emph{realization space}
$\mathcal{R}(\tilde{G},\varphi,\bm{\ell})$ and its quotient by the Euclidean isometries
$\mathcal{C}(\tilde{G},\varphi,\bm{\ell}) = \mathcal{R}(\tilde{G},\varphi,\bm{\ell})/\Euc(2)$
to be the configuration space. A realization is \emph{rigid} if it is isolated in the
configuration space and otherwise \emph{flexible}.

As the combinatorial model for reflection frameworks it will be more convenient
to use colored graphs.  A \emph{colored graph} $(G,\bgamma)$ is a finite,
directed%
\footnote{For the group $\Z/2\Z$, the orientation of the edges do not play a role, but we give the standard
definition for consistency.} %
graph $G$, with an assignment $\bgamma = (\gamma_{ij})_{ij\in E(G)}$
of an element of a group $\Gamma$ to each edge.  In this paper $\Gamma$ is always
$\Z/2\Z$. There is a standard dictionary
\cite[Section 9]{MT11} associating $(\tilde{G},\varphi)$ with a colored
graph $(G,\bgamma)$: $G$ is the quotient of $\tilde{G}$ by $\Gamma$, and the
colors encode the covering map via a natural map $\rho : \pi_1(G,b) \to \Gamma$.
In this setting, the choice of base vertex does not matter, and indeed, we may
define $\rho : \HH_1(G, \Z)\to \Z/2\Z$ and obtain the same theory.

\subsection{Main Theorem}
We can now state the main result of this paper.
\begin{theorem}[\reflectionlaman]\theolab{reflection-laman}
A generic reflection framework is minimally rigid if and only if its associated colored
graph is reflection-Laman.
\end{theorem}
The \emph{reflection-Laman graphs}
appearing in the statement are defined in \secref{matroid}.  Genericity has
its standard meaning from algebraic geometry: the set of non-generic reflection
frameworks is a measure-zero algebraic set, and a small \emph{geometric} perturbation
of a non-generic reflection framework yields a generic one.

\subsection{Infinitesimal rigidity and direction networks}
As in all known proofs of ``Maxwell-Laman-type'' theorems such as \theoref{reflection-laman},
we give a combinatorial characterization of a linearization of the problem known as
\emph{infinitesimal rigidity}.  To do this, we use a \emph{direction network} method
(cf. \cite{W88,ST10,MT10,MT11}).  A \emph{reflection direction network} $(\tilde{G},\varphi,\vec d)$ is a
symmetric graph, along with an assignment of a \emph{direction} $\vec d_{ij}$ to each edge.  The
\emph{realization space} of a direction network is the set of solutions $\tilde{G}(\vec p)$ to the system of
equations:
\begin{eqnarray}
\eqlab{dn-realization1}
\iprod{\vec p_j - \vec p_i}{\vec d^{\perp}_{ij}} = 0 & \qquad \text{for all edges $ij\in E(\tilde{G})$} \\
\eqlab{dn-realization2}
\vec p_{\varphi(\gamma)\cdot i} = \Phi(\gamma)\cdot\vec p_i & \qquad
\text{for all $\gamma\in \Z/2\Z$ and $i\in V(\tilde{G})$}
\end{eqnarray}
where the $\Z/2\Z$-action $\Phi$ on the plane is by reflection through the $y$-axis.
A reflection direction network is
determined by assigning a direction to each edge of the colored quotient graph $(G,\bgamma)$
of $(\tilde{G},\varphi)$ (cf. \cite[Lemma 17.2]{MT11}).
Since all the direction networks in this paper are
reflection direction networks, we will refer to them simply as ``direction networks''
to keep the terminology manageable.  A realization of a direction network is
\emph{faithful} if none of the edges of its graph have coincident endpoints and
\emph{collapsed} if all the endpoints are coincident.

A basic fact in the theory of finite planar frameworks \cite{W88,ST10,DMR07}
is that, if a direction network has faithful realizations, the dimension of the realization
space is equal to that of the space of infinitesimal motions of a generic framework with the same
underlying graph.  In \cite{MT10,MT11}, we adapted this idea to the symmetric case when all the
symmetries act by rotations and translations.

As discussed in \cite[Section 1.8]{MT11}, this so-called ``parallel redrawing trick''%
\footnote{This terminology comes from the engineering community, in which the basic idea has been folklore for
quite some time.} %
described above does \emph{not} apply verbatim to reflection frameworks.
Thus, we rely on the somewhat technical (cf. \cite[Theorem B]{MT10},
\cite[Theorem 2]{MT11}) \theoref{direction-network}, which we state after giving an important
definition.

Let $(\tilde{G},\varphi,\vec d)$ be a direction network and define $(\tilde{G},\varphi,\vec d^{\perp})$
to be the direction network with $(\vec d^\perp)_{ij} = (\vec d_{ij})^\perp$.  These two direction
networks form a \emph{special pair} if:
\begin{itemize}
\item $(\tilde{G},\varphi,\vec d)$ has a faithful realization.
\item $(\tilde{G},\varphi,\vec d^\perp)$ has only collapsed realizations.
\end{itemize}
\begin{theorem}[\linkeddirectionnetworks]\theolab{direction-network}
Let $(G,\bgamma)$ be a colored graph with $n$ vertices, $2n-1$ edges, and lift $(\tilde{G},\varphi)$.
Then there are directions $\vec d$ such that the direction networks $(\tilde{G},\varphi,\vec d)$
and $(\tilde{G},\varphi,\vec d^\perp)$ are a special pair if and only if $(G,\bgamma)$ is reflection-Laman.
\end{theorem}
Briefly, we will use \theoref{direction-network} as follows: the faithful realization of
$(\tilde{G},\varphi,\vec d)$ gives a symmetric immersion of the graph $\tilde{G}$ that can be interpreted
as a framework, and the fact that $(\tilde{G},\varphi,\vec d^\perp)$ has only collapsed realizations
will imply that the only symmetric infinitesimal motions of this framework correspond to translation
parallel to the reflection axis.

\subsection{Notations and terminology}  In this paper, all graphs $G=(V,E)$ may be multi-graphs.  Typically,
the number of vertices, edges, and connected components are denoted by $n$, $m$, and $c$, respectively.
The notation for a colored graph is $(G,\bgamma)$, and a symmetric graph with a free $\Z/2\Z$-action
is denoted by $(\tilde{G},\varphi)$.  If  $(\tilde{G},\varphi)$ is the lift of $(G,\bgamma)$, we
denote the fiber over a vertex $i\in V(G)$ by $\tilde{i}_\gamma$, with $\gamma\in \Z/2\Z$, and the fiber over a directed
edge $ij$ with color $\gamma_{ij}$ by $\tilde{i}_\gamma \tilde{j}_{\gamma+\gamma_{ij}}$.

We also use \emph{$(k,\ell)$-sparse graphs} \cite{LS08} and their generalizations.  For a graph $G$, a
\emph{$(k,\ell)$-basis} is a maximal $(k,\ell)$-sparse subgraph; a \emph{$(k,\ell)$-circuit} is an edge-wise
minimal subgraph that is not $(k,\ell)$-sparse; and a \emph{$(k,\ell)$-component} is a maximal subgraph
that has a spanning $(k,\ell)$-graph.

Points in $\R^2$ are denoted by $\vec p_i = (x_i,y_i)$, indexed sets of points by $\vec p = (\vec p_i)$,
and direction vectors by $\vec d$ and $\vec v$.  Realizations
of a reflection direction network $(\tilde{G},\varphi,\vec d)$ are written as $\tilde{G}(\vec p)$, as are realizations
of abstract reflection frameworks.  Context will always make clear the type of realization under consideration.

\subsection{Acknowledgements}
LT is supported by the European Research Council under the European Union's Seventh Framework
Programme (FP7/2007-2013) / ERC grant agreement no 247029-SDModels. JM is supported by NSF
CDI-I grant DMR 0835586.

\section{Reflection-Laman graphs} \seclab{matroid}
In this short section we introduce the combinatorial families of sparse colored graphs we use.

\subsection{The map $\rho$}
Let $(G,\bgamma)$ be a $\Z/2\Z$-colored graph.  Since all the colored graphs in this paper have $\Z/2\Z$ colors,
from now on we make this assumption and write simply ``colored graph''.  We recall two key definitions
from \cite{MT11}.

The map $\rho : \HH_1(G, \Z)\to \Z/2\Z$ is defined on cycles by adding up the colors on the edges.  (The directions of the
edges don't matter for $\Z/2\Z$ colors.  Similarly, neither does the traversal order.)  As the notation suggests,
$\rho$ extends to a homomorphism from $\HH_1(G, \Z)$ to $\Z/2\Z$, and it is well-defined even if $G$ is not connected.

\subsection{Reflection-Laman graphs}
Let $(G,\bgamma)$ be a colored graph with $n$ vertices and $m$ edges.  We define $(G,\bgamma)$ to be a
\emph{reflection-Laman graph} if: the number of edges $m=2n-1$, and for all subgraphs $G'$, spanning $n'$
vertices, $m'$ edges, $c'$ connected components with non-trivial $\rho$-image and $c'_0$ connected components
with trivial $\rho$-image
\begin{equation}\eqlab{cone-laman}
m'\le 2n' - c' - 3c'_0
\end{equation}
This definition is equivalent to that of \emph{cone-Laman graphs} in \cite[Section 15.4]{MT11}.
The underlying graph $G$ of a reflection-Laman graph is a $(2,1)$-graph.

\subsection{Ross graphs and circuits}
Another family we need is that of \emph{Ross graphs}
(see \cite{BHMT11} for an explanation of the terminology).  These are colored graphs with $n$ vertices,
$m = 2n - 2$ edges, satisfying the sparsity counts
\begin{equation}\eqlab{ross}
m'\le 2n' - 2c' - 3c'_0
\end{equation}
using the same notations as in \eqref{cone-laman}.  In particular, Ross graphs $(G,\bgamma)$
have as their underlying graph, a $(2,2)$-graph $G$, and are thus connected \cite{LS08}.

A \emph{Ross-circuit}%
\footnote{The matroid of Ross graphs has more circuits, but these are the ones
we are interested in here.  See \secref{reflection-22}.} %
is a colored graph that becomes a Ross graph after removing \emph{any} edge.  The underlying
graph $G$ of a Ross-circuit $(G,\bgamma)$ is a $(2,2)$-circuit, and these are also known
to be connected \cite{LS08}, so, in particular, a Ross-circuit has $c'_0=0$, and
thus satisfies \eqref{cone-laman} on the whole graph.  Since \eqref{cone-laman} is
always at least \eqref{ross}, we see that every Ross-circuit is reflection-Laman.

Because reflection-Laman graphs are $(2,1)$-graphs and subgraphs that are $(2,2)$-sparse
are, in addition, Ross-sparse, we get the following structural result.
\begin{prop}[\xyzzy][{\cite[Proposition 5.1]{MT12},\cite[Lemma 11]{BHMT11}}]\proplab{ross-circuit-decomp}
Let $(G,\bgamma)$ be a reflection-Laman graph.  Then each $(2,2)$-component of $G$ contains at most one
Ross-circuit, and in particular, the Ross-circuits in $(G,\bgamma)$ are vertex disjoint.
\end{prop}

\subsection{Reflection-$(2,2)$ graphs}\seclab{reflection-22}
The next family of graphs we work with is new.  A colored graph $(G,\bgamma)$ is defined to be
a \emph{reflection-$(2,2)$} graph, if it has $n$ vertices, $m=2n-1$ edges, and satisfies the
sparsity counts
\begin{equation} \eqlab{ref22a}
m' \le 2n' - c' - 2c'_0
\end{equation}
using the same notations as in \eqref{cone-laman}.

The relationship between Ross graphs and reflection-$(2,2)$ graphs we will need is:
\begin{prop} \proplab{ross-adding}
Let $(G,\bgamma)$ be a Ross-graph.  Then for either
\begin{itemize}
\item an edge $ij$ with any color where $i \neq j$
\item or a self-loop $\ell$ at any vertex $i$ colored by $1$
\end{itemize}
the graph $(G+ij,\bgamma)$ or $(G+\ell,\bgamma)$ is reflection-$(2,2)$.
\end{prop}
\begin{proof}
Adding $ij$ with any color to a Ross $(G,\bgamma)$ creates either a Ross-circuit, for which
$c'_0=0$ or a Laman-circuit with trivial $\rho$-image.  Both of these types of
graph meet this count, and so the whole of $(G+ij,\bgamma)$ does as well.
\end{proof}
It is easy to see that every reflection-Laman graph is a reflection-$(2,2)$ graph.  The
converse is not true.
\begin{prop}\proplab{reflection-laman-vs-reflection-22}
A colored graph $(G,\bgamma)$ is a reflection-Laman graph if and only if it is a
reflection-$(2,2)$ graph and no subgraph with trivial $\rho$-image is a $(2,2)$-block.\eop
\end{prop}
Let $(G,\bgamma)$ be a reflection-Laman graph, and let $G_1,G_2,\ldots,G_t$ be the Ross-circuits
in $(G,\bgamma)$.  Define the \emph{reduced graph} $(G^*,\bgamma)$ of $(G,\bgamma)$ to be the
colored graph obtained by contracting each $G_i$, which is not already a single vertex with a self-loop (this is necessarily
colored $1$), into a new vertex $v_i$, removing any self-loops
created in the process, and then adding a new self-loop with color $1$ to each of the $v_i$.
By \propref{ross-circuit-decomp} the reduced graph is well-defined.
\begin{prop}\proplab{reduced-graph}
Let $(G,\bgamma)$ be a reflection-Laman graph.  Then its reduced graph is a reflection-$(2,2)$ graph.
\end{prop}
\begin{proof}
Let $(G,\bgamma)$ be a reflection-Laman graph with $t$ Ross-circuits with vertex sets $V_1,\ldots,V_t$.
By \propref{ross-circuit-decomp}, the $V_i$ are all
disjoint. Now select a Ross-basis $(G',\bgamma)$ of $(G,\bgamma)$.  The graph $G'$
is also a $(2,2)$-basis of $G$, with $2n-1 - t$ edges, and each of the $V_i$ spans a $(2,2)$-block
in $G'$. The $(k,\ell)$-sparse graph
Structure Theorem \cite[Theorem 5]{LS08} implies that contracting each of the $V_i$ into a new vertex $v_i$
and discarding any self-loops created, yields a $(2,2)$-sparse graph $G^+$ on $n^+$ vertices and $2n^+ - 1 - t$
edges.  It is then easy to check that adding a self-loop colored $1$ at each of the $v_i$ produces
a colored graph satisfying the reflection-$(2,2)$ counts \eqref{ref22a} with exactly
$2n^+ -1$ edges.  Since this is the reduced graph, we are done.
\end{proof}

\subsection{Decomposition characterizations}
A \emph{map-graph} is a graph with exactly one cycle per connected component.  A \emph{reflection-$(1,1)$} graph
is defined to be a colored graph $(G,\bgamma)$ where $G$, taken as an undirected graph, is a map-graph and
the $\rho$-image of each connected component is non-trivial.
\begin{lemma}\lemlab{reflection-22-decomp}
Let $(G,\bgamma)$ be a colored graph.  Then $(G,\bgamma)$ is a reflection-$(2,2)$ graph if and only
if it is the union of a spanning tree and a reflection-$(1,1)$ graph.
\end{lemma}
\begin{proof}
By \cite[Lemma 15.1]{MT11}, reflection-$(1,1)$ graphs are equivalent to graphs satisfying
\begin{equation} \eqlab{ref11a}
m' \le n' - c'_0
\end{equation}
for every subgraph $G'$.  Thus, \eqref{ref22a} is
\begin{equation}\eqlab{ref22redux}
m' \le (n' - c'_0) + (n' - c' - c'_0)
\end{equation}
The second term in \eqref{ref22redux} is well-known to be the rank function of the graphic matroid, and
the Lemma follows from the Edmonds-Rota construction \cite{ER66} and the Matroid Union Theorem.
\end{proof}
In the next section, it will be convenient to use this slight refinement of \lemref{reflection-22-decomp}.
\begin{prop}\proplab{reflection-22-nice-decomp}
Let $(G,\bgamma)$ be a reflection-$(2,2)$ graph.  Then there is a coloring $\bgamma'$ of the edges of
$G$ such that:
\begin{itemize}
\item The $\rho$-image of every subgraph in $(G,\bgamma')$ is the same as in $(G,\bgamma)$.
\item There is a decomposition of $(G,\bgamma')$ as in \lemref{reflection-22-decomp}
in which the spanning tree has all edges colored by the identity.
\end{itemize}
\end{prop}
\begin{proof}
It is shown in \cite[Lemma 2.2]{MT10} that $\rho$ is determined by its image on a
homology basis of $G$.  Thus, we may start with an arbitrary decomposition of $(G,\bgamma)$
into a spanning tree $T$ and a reflection-$(1,1)$ graph $X$, as provided by
\lemref{reflection-22-decomp}, and define $\bgamma'$ by coloring the edges of $T$ with the identity
and the edges of $X$ with the $\rho$-image of their fundamental cycle in $T$ in $(G,\bgamma)$.
\end{proof}
\propref{reflection-22-nice-decomp} has the following re-interpretation in terms of the symmetric
lift $(\tilde{G},\varphi)$:
\begin{prop}\proplab{reflection-laman-decomp-lift}
Let $(G,\bgamma)$ be a reflection-$(2,2)$ graph.  Then for a decomposition, as provided
by \propref{reflection-22-nice-decomp}, into a spanning tree $T$ and a reflection-$(1,1)$
graph $X$:
\begin{itemize}
\item Every edge $ij\in T$ lifts to the two edges $i_0j_0$ and $i_1j_1$.
(In other words, the vertex representatives in the lift all lie	in a single connected
component of the lift of $T$.)
\item Each connected component of $X$ lifts to a connected graph.
\end{itemize}
\end{prop}

\section{Special pairs of reflection direction networks} \seclab{direction-network}
We recall, from the introduction, that for reflection direction networks, $\Z/2\Z$ acts on the plane
by reflection through the $y$-axis, and in the rest of this section $\Phi(\gamma)$ refers to this action.

\subsection{The colored realization system}
The system of equations \eqref{dn-realization1}--\eqref{dn-realization2}
defining the realization space of a reflection direction network
$(\tilde{G},\varphi,\vec d)$ is linear, and
as such has a well-defined dimension.  Let $(G,\bgamma)$ be the
colored quotient graph of $(\tilde{G},\varphi)$.

To be realizable at all, the directions on the edges in the fiber over $ij\in E(G)$
need to be reflections of each other.  Thus, we see that the realization system
is canonically identified with the solutions to the system:
\begin{eqnarray}\eqlab{colored-system}
\iprod{\Phi(\gamma_{ij})\cdot\vec p_j - \vec p_i}{\vec d_{ij}} = 0 & \qquad
\text{for all edges $ij\in E(G)$}
\end{eqnarray}
From now on, we will implicitly switch between the two formalisms when it is
convenient.

\subsection{Genericity}
Let $(G,\bgamma)$ be a colored graph with $m$ edges.  A statement about direction networks
$(\tilde{G},\varphi,\vec d)$ is \emph{generic} if it holds on the complement of a
proper algebraic subset of the possible direction assignments, which is canonically identified
with $\R^{2m}$. Some facts about generic statements that
we use frequently are:
\begin{itemize}
\item Almost all direction assignments are generic.
\item If a set of directions is generic, then so are all sufficiently
small perturbations of it.
\item If two properties are generic, then their intersection is as well.
\item The maximum rank of \eqref{colored-system} is a generic property.
\end{itemize}

\subsection{Direction networks on Ross graphs}
We first characterize the colored graphs for which generic direction networks have
strongly faithful realizations.  A realization is \emph{strongly faithful} if no
two vertices lie on top of each other.  This is a stronger condition than simply being
faithful which only requires that edges not be collapsed.
\begin{prop}\proplab{ross-realizations}
A generic direction network  $(\tilde{G},\varphi,\vec d)$
has a unique, up to translation and scaling, strongly faithful realization
if and only if its associated colored graph is a Ross graph.
\end{prop}
To prove \propref{ross-realizations} we expand upon the method from \cite[Section 20.2]{MT11},
and use the following proposition.
\begin{prop}\proplab{reflection-22-collapse}
Let $(G,\bgamma)$ be a reflection-$(2,2)$ graph.  Then a generic direction network on the
symmetric lift $(\tilde{G},\varphi)$ of $(G,\bgamma)$ has only collapsed realizations.
\end{prop}
Since the proof of \propref{reflection-22-collapse} requires a detailed construction,
we first show how it implies \propref{ross-realizations}.
\subsection{Proof that \propref{reflection-22-collapse} implies \propref{ross-realizations}}
Let $(G,\bgamma)$ be a Ross graph, and assign directions $\vec d$ to the edges of $G$ such that,
for any extension $(G+ij,\bgamma)$ of $(G,\bgamma)$ to a reflection-$(2,2)$ graph as
in \propref{ross-adding}, $\vec d$ can be extended to a set of directions that is generic
in the sense of \propref{reflection-22-collapse}.  This is possible because there are a finite
number of such extensions.

For this choice of $\vec d$, the realization space of the direction network $(\tilde{G},\varphi,\vec d)$
is $2$-dimensional.  Since solutions to \eqref{colored-system} may be scaled or translated in the vertical
direction, all solutions to $(\tilde{G},\varphi,\vec d)$ are related by scaling and translation.  It
then follows that a pair of vertices in the fibers over $i$ and $j$ are either distinct from each other
in all non-zero solutions to \eqref{colored-system} or always coincide.  In the latter case, adding the
edge $ij$ with any direction does not change the dimension of the solution space, no matter
what direction we assign to it.  It then follows that the solution spaces of generic
direction networks on $(\tilde{G},\varphi,\vec d)$ and $(\widetilde{G+ij},\varphi,\vec d)$ have
the same dimension, which is a contradiction by \propref{reflection-22-collapse}.
\eop

\subsection{Proof of \propref{reflection-22-collapse}}
It is sufficient to construct a specific set of directions with this property.  The
rest of the proof gives such a construction and verifies that all the solutions
are collapsed.  Let $(G,\bgamma)$ be a reflection-$(2,2)$ graph.

\paragraph{Combinatorial decomposition}
We apply \propref{reflection-22-nice-decomp} to decompose $(G,\bgamma)$ into a
spanning tree $T$ with all colors the identity and a reflection-$(1,1)$ graph $X$.  For
now, we further assume that $X$ is connected.

\paragraph{Assigning directions}
Let $\vec v$ be a direction vector that is not
horizontal or vertical.  For each edge $ij\in T$, set $\vec d_{ij} = \vec v$.
Assign all the edges of $X$ the vertical
direction.  Denote by $\vec d$ this assignment of directions.
\begin{figure}[htbp]
\centering
\includegraphics[width=.3\textwidth]{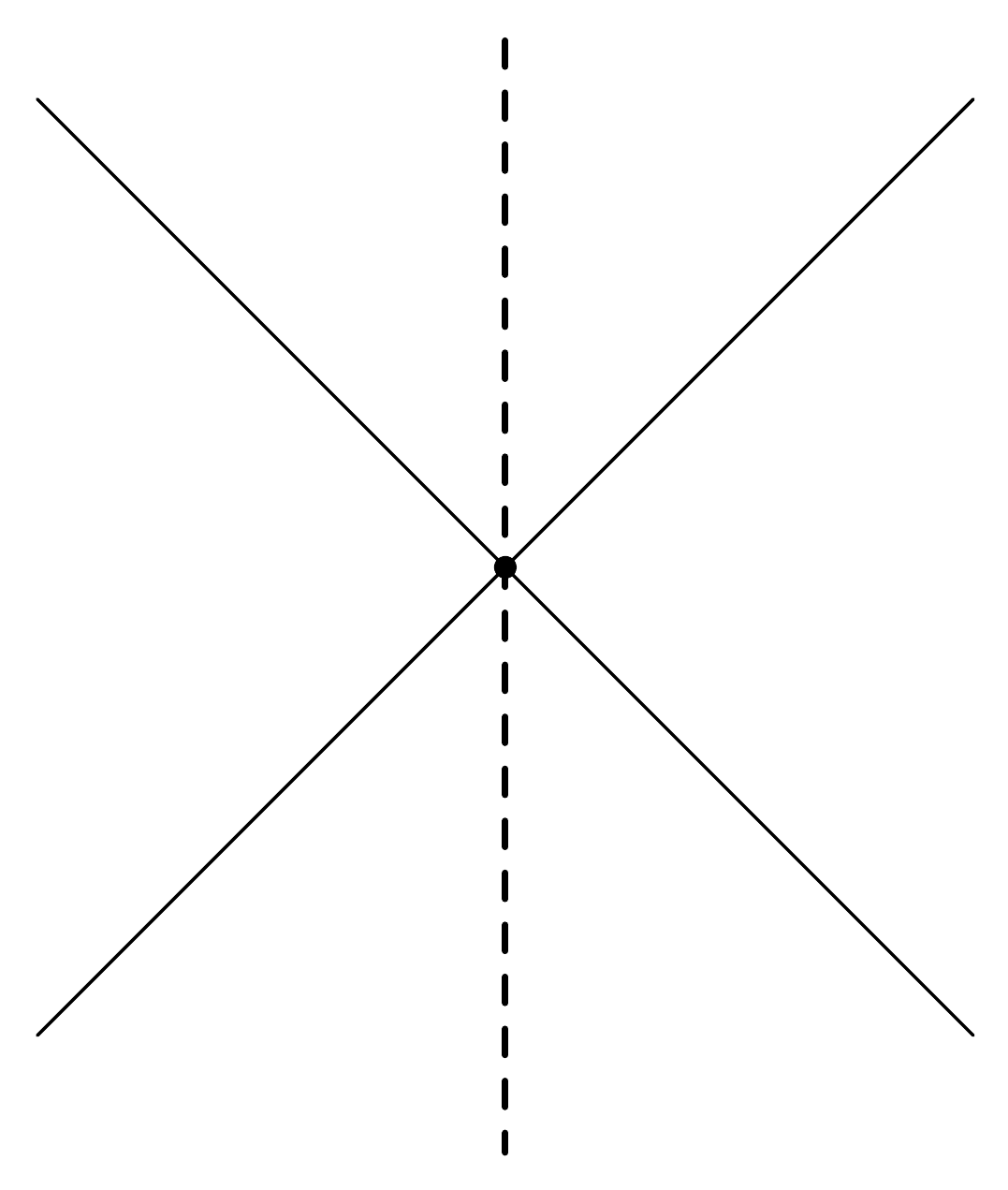}
\caption{Schematic of the proof of \propref{reflection-22-collapse}: the $y$-axis
is shown as a dashed line.  The directions on the edges of the lift of the tree
$T$ force all the vertices to be on one of the two lines meeting at the $y$-axis,
and the directions on the reflection-$(1,1)$ graph $X$ force all the vertices to
be on the $y$-axis.}
\label{fig:ref-22-collapse}
\end{figure}
\paragraph{All realizations are collapsed}
We now show that the only realizations of $(\tilde{G},\varphi,\vec d)$ have
all vertices on top of each other.  By \propref{reflection-laman-decomp-lift}
$T$ lifts to two copies of itself, in $\tilde{G}$.  It then follows from
the connectivity of $T$ and the construction of $\vec d$ that, in any
realization, there is a line $L$ with direction $\vec v$ such that every vertex
of $\tilde{G}$ must lie on $L$ or its reflection.  Since the vertical direction
is preserved by reflection, the connectivity of the lift of $X$, again from
\propref{reflection-laman-decomp-lift}, implies that every vertex of $\tilde{G}$
lies on a single vertical line, which must be the $y$-axis by reflective symmetry.

Thus, in any realization of $(\tilde{G},\varphi,\vec d)$ all the vertices lie
at the intersection of $L$, the reflection of $L$
through the $y$-axis and the $y$-axis itself.  This is a single point, as
desired.  \figref{ref-22-collapse} shows a schematic of this argument.

\paragraph{$X$ does not need to be connected}
Finally, we can remove the assumption that $X$  was connected by repeating the argument
for each connected component of $X$ separately.
\eop

\subsection{Special pairs for Ross-circuits}
The full \theoref{direction-network} will reduce to the case of a Ross-circuit.
\begin{prop}\proplab{ross-circuit-pairs}
Let $(G,\bgamma)$ be a Ross circuit with lift $(\tilde{G},\varphi)$.  Then there
is an edge $i'j'$ such that, for a generic direction network $(\tilde{G'},\varphi,\vec d')$ with colored graph
$(G-i'j',\bgamma)$:
\begin{itemize}
\item The solution space of  $(\tilde{G'},\varphi,\vec d')$ induces a well-defined direction
$\vec d_{ij}$ between $i$ and $j$, yielding an assignment of directions $\vec d$ to the edges
of $G$.
\item The direction networks $(\tilde{G},\varphi,\vec d)$
and $(\tilde{G},\varphi,(\vec d)^\perp)$ are a special pair.
\end{itemize}
\end{prop}
Before giving the proof, we describe the idea.  We are after sets of directions that
lead to faithful realizations of Ross-circuits.  By \propref{reflection-22-collapse},
these directions must be non-generic.  A natural way to obtain such a set of directions
is to discard an edge $ij$ from the colored quotient graph,
apply \propref{ross-realizations} to obtain a generic set of directions $\vec d'$ with a
strongly faithful realization $\tilde{G}'(\vec p)$, and then simply set the directions on the
edges in the fiber over $ij$ to be the difference vectors between the points.

\propref{ross-realizations} tells us that this procedure induces a well-defined
direction for the edge $ij$, allowing us to extend $\vec d'$ to $\vec d$ in a controlled way.
However, it does \emph{not} tell us that rank
of $(\tilde{G},\varphi,\vec d)$ will rise when the directions are turned
by angle $\pi/2$, and this seems hard to do directly.  Instead,
we construct a set of directions $\vec d$ so that $(\tilde{G},\varphi,\vec d)$ is rank deficient
and has faithful realizations, and $(\tilde{G},\varphi,\vec d^\perp)$ is generic.
Then we make a perturbation argument to show the existence of a special pair.

The construction we use is, essentially, the one used in the proof of \propref{reflection-22-collapse}
but turned through angle $\pi/2$.  The key geometric insight is that horizontal edge directions
are preserved by the reflection, so the ``gadget'' of a line and its reflection crossing on the $y$-axis, as in
\figref{ref-22-collapse}, degenerates to just a single line.

\subsection{Proof of \propref{ross-circuit-pairs}}
Let $(G,\bgamma)$ be a Ross-circuit; recall that this implies that $(G,\bgamma)$ is a
reflection-Laman graph.

\paragraph{Combinatorial decomposition}
We decompose $(G,\bgamma)$ into a spanning tree $T$ and a
reflection-$(1,1)$ graph $X$ as in \propref{reflection-laman-decomp-lift}.  In particular,
we again have all edges in $T$ colored by the identity.  For now,
we \emph{assume that $X$ is connected}, and we fix $i'j'$ to be an edge that is on the cycle in $X$
with $\gamma_{i'j'}\neq 0$; such an edge must exist by the hypothesis that $X$ is
reflection-$(1,1)$. Let $G' = G\setminus i'j'$.  Furthermore, let $T_0$ and $T_1$
be the two connected components of the lift of $T$.  For a vertex $i \in G$, we denote
the lift in $T_0$ by $i_0$ and the lift in $T_1$ by $i_1$.  We similarly denote the
lifts of $i'$ and $j'$ by $i_0', i_1'$ and $j_0', j_1'$.

\paragraph{Assigning directions}
The assignment of directions is as follows: to the edges of $T$, we assign a direction
$\vec v$ that is neither vertical nor horizontal.  To the edges of $X$
we assign the horizontal direction.  Define the resulting direction network to be
$(\tilde{G},\varphi,\vec d)$, and the direction network induced on the lift of $G'$ to be
$(\tilde{G'},\varphi,\vec d)$.

\paragraph{The realization space of $(\tilde{G},\varphi,\vec d)$}
\figref{ross-circuit-special-pair} contains a schematic picture of the arguments that
follow.
\begin{lemma}\lemlab{RC-proof-1}
The realization space of $(\tilde{G},\varphi,\vec d)$ is $2$-dimensional
and parameterized by exactly one representative in the fiber over the
vertex $i$ selected above.
\end{lemma}
\begin{proof}
In a manner similar to
the proof of \propref{reflection-22-collapse}, the directions on the edges of $T$ force every
vertex to lie either on a line $L$ in the direction $\vec v$ or its reflection.  Since the lift
of $X$ is connected, we further conclude that all the vertices lie on a single horizontal line.
Thus, all the points $\vec p_{j_0}$ are at the intersection of the same horizontal line and $L$
or its reflection. These determine the locations of the $\vec p_{j_1}$, so the
realization space is parameterized by the location of $\vec p_{i'_0}$.
\end{proof}
Inspecting the argument more closely, we find that:
\begin{lemma}
In any realization $\tilde{G}(\vec p)$ of $(\tilde{G},\varphi,\vec d)$,
all the $\vec p_{j_0}$ are equal and all the $\vec p_{j_1}$ are equal.
\end{lemma}
\begin{proof}
Because the colors on the edges of $T$ are all zero, it lifts to two copies of itself,
one of which spans the vertex set $\{\tilde{j_0} : j\in V(G)\}$ and one which spans
$\{\tilde{j_1} : j\in V(G)\}$.  It follows that in a realization, we have all the
$\vec p_{j_0}$ on $L$ and the $\vec p_{j_1}$ on the
reflection of $L$.
\end{proof}
In particular, because the color $\gamma_{i'j'}$ on the edge $i'j'$ is $1$, we obtain the following.
\begin{lemma}\lemlab{RC-proof-5}
The realization space of $(\tilde{G},\varphi,\vec d)$ contains points
where the fiber over the edge $i'j'$ is not collapsed.
\end{lemma}
\begin{figure}[htbp]
\centering
\includegraphics[width=.3\textwidth]{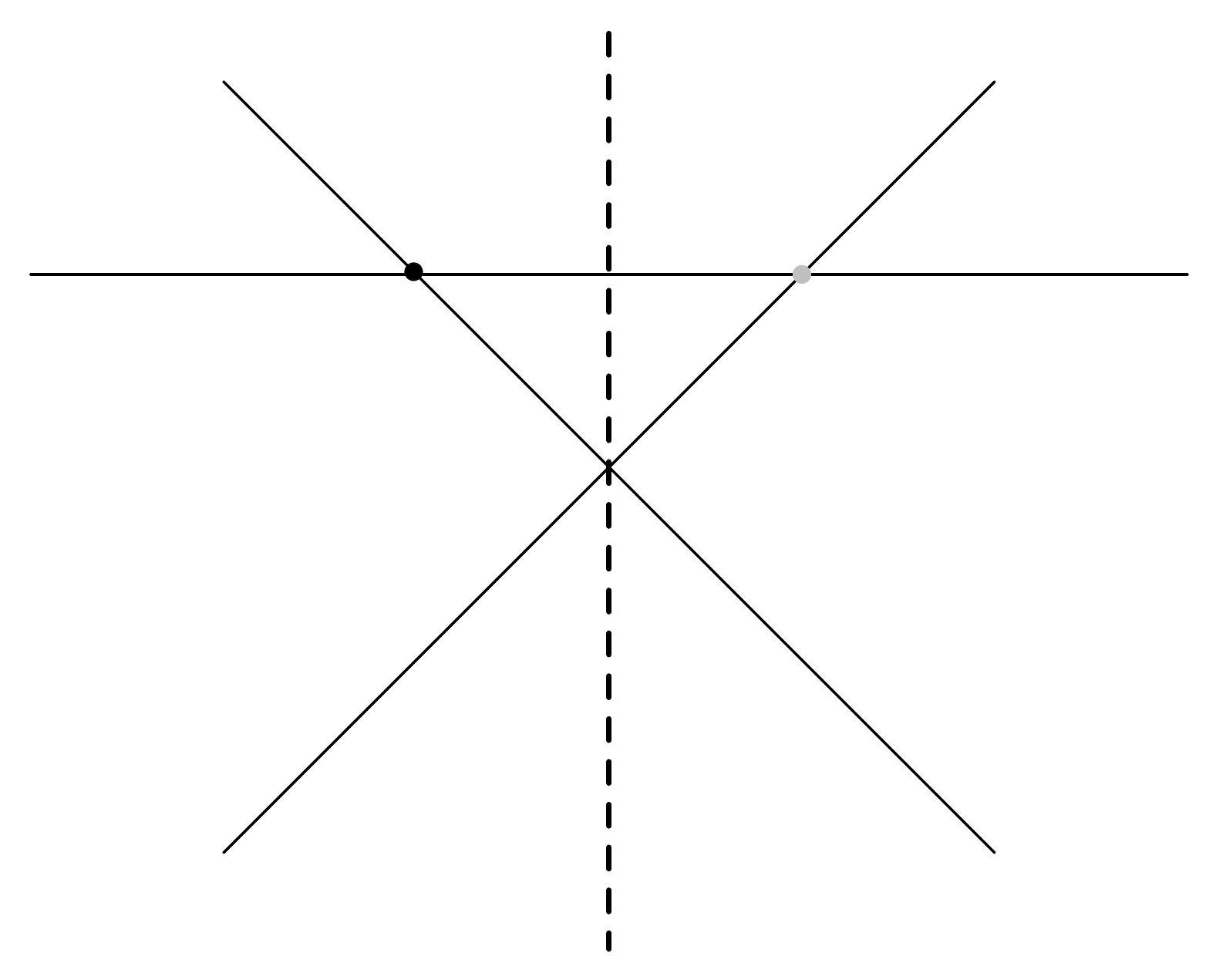}
\caption{Schematic of the proof of \propref{ross-circuit-pairs}: the $y$-axis
is shown as a dashed line.  The directions on the edges of the lift of the tree
$T$ force all the vertices to be on one of the two lines meeting at the $y$-axis.
The horizontal directions on the connected reflection-$(1,1)$ graph $X$ force the
point $\vec p_{j_0}$ to be at the intersection marked by the black dot and
$\vec p_{j_1}$ to be at the intersection marked by the gray one.}
\label{fig:ross-circuit-special-pair}
\end{figure}
\paragraph{The realization space of $(\tilde{G}',\varphi,\vec d)$}
The conclusion of \lemref{RC-proof-1} implies that the realization
system for $(\tilde{G},\varphi,\vec d)$ is rank deficient by one.
Next we show that removing the edge $i'j'$ results in a
direction network that has full rank on the colored graph $(G',\bgamma)$.
\begin{lemma}\lemlab{RC-proof-2}
The realization space of $(\tilde{G},\varphi,\vec d)$ is canonically identified
with that of $(\tilde{G}',\varphi,\vec d)$.
\end{lemma}
\begin{proof}
In the proof of \lemref{RC-proof-1}, that $X$ lifts to a connected subgraph of $\tilde{G}$
was not essential.  Because a horizontal line is preserved by the reflection,
realizations will take on the same structure provided that $X$ lifts to a subgraph
with two connected components.  Removing $i'j'$
from $X$ leaves a graph $X'$ with this property since $X'$ is a tree.

It follows that the equation corresponding to the edge $i'j'$ in \eqref{colored-system}
was dependent.
\end{proof}

\paragraph{The realization space of $(\tilde{G},\varphi,\vec d^\perp)$}
Next, we consider what happens when we turn all the directions by $\pi/2$.
\begin{lemma}\lemlab{RC-proof-3}
The realization space of $(\tilde{G},\varphi,\vec d^\perp)$ has only collapsed solutions.
\end{lemma}
\begin{proof}
This is exactly the construction used to prove \propref{reflection-22-collapse}.
\end{proof}

\paragraph{Perturbing $(\tilde{G},\varphi,\vec d)$}
To summarize what we have shown so far:
\begin{itemize}
\item[(a)] $(\tilde{G},\varphi,\vec d)$ has a $2$-dimensional realization space parameterized
by $\vec p_{i'_0}$ and identified with that of a full-rank direction network on  the Ross graph
$(G',\bgamma)$.
\item[(b)] There are points $\tilde{G}(\vec p)$ in this realization space where
$\vec p_{i'_0}\neq \vec p_{j'_1}$.
\item[(c)] $(\tilde{G},\varphi,\vec d)$ has a $1$-dimensional realization space containing only collapsed
solutions.
\end{itemize}
What we have not shown is that the realization space of $(\tilde{G},\varphi,\vec d)$
has \emph{faithful} realizations, since the ones we constructed all have many
coincident vertices.  \propref{ross-realizations} will imply the rest of the
theorem, provided that the above properties hold for any small perturbation of
$\vec d$, since some small perturbation of \emph{any} assignment of directions
to the edges of $(G',\bgamma)$ has only faithful realizations.
\begin{lemma}\lemlab{RC-proof-4}
Let $\vec{ \hat d'}$ be a perturbation of the directions $\vec d'$ on the edges of $G'$.  If $\vec{ \hat d'}$
is sufficiently close to $\vec d'$ ,
then there are realizations of the direction network
$(\tilde{G}',\varphi,\vec{ \hat d'})$ such that $\vec p_{i'_0}\neq \vec p_{j'_1}$.
\end{lemma}
\begin{proof}
The realization space is parameterized by $\vec p_{i'_0}$, and so $\vec p_{j'_1}$ varies
continuously with the directions on the edges and $\vec p_{i'_0}$.  Since there are realizations of $(\tilde{G}', \varphi, \vec d)$
with $\vec p_{i_0} \neq \vec p_{j_1}$, the Lemma follows.
\end{proof}
\lemref{RC-proof-4} implies that any sufficiently small perturbation of the directions assigned to the
edges of $G'$ gives a direction network that induces a well-defined direction on the edge $i'j'$ which
is itself a small perturbation of $\vec d_{i'j'}$.  Since the ranks of $(\tilde{G'},\varphi,\vec d')$
and  $(\tilde{G},\varphi,\vec d^\perp)$ are stable under small perturbations, this implies that
we can perturb $\vec d'$ to a $\vec{\hat d'}$ that is generic in the sense of \propref{ross-realizations},
while preserving faithful realizability of $(\tilde{G},\varphi,\hat{\vec d})$ and full rank of the
realization system for $(\tilde{G},\varphi,\hat{\vec d}^\perp)$.  The Proposition is proved for when $X$ is
connected.

\paragraph{$X$ need not be connected}
The proof is then complete once we remove the additional assumption that $X$ was
connected.  Let $X$ have connected
components $X_1, X_2,\ldots,X_c$.  For each of the $X_i$, we can identify an edge
$(i'j')_k$ with the same properties as $i'j'$ above.

Assign directions to the tree $T$ as above.
For $X_1$, we assign directions exactly as above.  For each of the $X_k$ with $k\ge 2$,
we assign the edges of $X_k\setminus (i'j')_k$ the horizontal direction and $(i'j')_k$
a direction that is a small perturbation of horizontal.

With this assignment $\vec d$ we see that
for any realization of $(\tilde{G},\varphi,\vec d)$, each of the $X_k$, for $k\ge 2$
is realized as completely collapsed to a single point
at the intersection of the line $L$ and the $y$-axis.  Moreover,
in the direction network on $\vec d^\perp$, the directions on these $X_i$ are a small
perturbation of the ones used on $X$ in the proof of \propref{reflection-22-collapse}.
From this is follows that, in any realization $(\tilde{G},\varphi,\vec d^\perp)$,
is completely collapsed and hence full rank.

We now see that this new set of directions has properties (a), (b), and (c) above
required for the perturbation argument.  Since that argument makes no reference
to the decomposition, it applies verbatim to the case where $X$ is disconnected.
\eop

\subsection{Proof of \theoref{direction-network}}
The easier direction to check is necessity.
\paragraph{The Maxwell-direction}
If $(G,\bgamma)$ is not reflection-Laman, then it contains either a
Laman-circuit with trivial $\rho$-image, or a violation of $(2,1)$-sparsity.  If there is a
Laman-circuit with trivial $\rho$-image, the Parallel Redrawing Theorem \cite[Theorem 4.1.4]{W96}
in the form \cite[Theorem 3]{ST10}  implies that this subgraph has no faithful realizations
for
$(G,\varphi,\vec d)$ only if it does in $(G,\varphi,\vec d^\perp)$ if rank-deficient.
A violation of $(2,1)$-sparsity implies that the realization system \eqref{colored-system}
of $(\tilde{G},\varphi,\vec d^\perp)$ has a dependency, since the realization space is always at
least $1$-dimensional.

\paragraph{The Laman direction}
Now let $(G,\bgamma)$ be a reflection-Laman graph and let $(G',\bgamma)$
be a Ross-basis of $(G,\bgamma)$.  For any edge $ij \notin G'$, adding it to
$G'$ induces a Ross-circuit which contains some edge $i'j'$ having the property
specified in \propref{ross-circuit-pairs}.  Note that $G' - ij +i'j'$ is again a Ross-basis.
We therefore can assume (after edge-swapping in this manner) for all $ij \notin G'$ that
$ij$ has the property from \propref{ross-circuit-pairs} in the Ross-circuit it induces.

We assign directions $\vec d'$ to the edges of $G'$ such that:
\begin{itemize}
\item The directions on each of the intersections of the Ross-circuits with $G'$ are generic in the sense
of \propref{ross-circuit-pairs}.
\item The directions on the edges of $G'$ that remain in the reduced graph $(G^*,\bgamma)$
are perpendicular to an assignment of directions on $G^*$ that is
generic in the sense of \propref{reflection-22-collapse}.
\item The directions on the edges of $G'$ are generic in the sense of \propref{ross-realizations}.
\end{itemize}
This is possible because the set of disallowed directions is the union of a finite number of
proper algebraic subsets in the space of direction assignments.  Extend to directions $\vec d$ on $G$
by assigning directions to the remaining edges as specified by \propref{ross-circuit-pairs}.  By construction,
we know that:
\begin{lemma}\lemlab{laman-1}
The direction network $(\tilde{G},\varphi,\vec d)$ has faithful realizations.
\end{lemma}
\begin{proof}
The realization space is identified with that of $(\tilde{G'},\varphi,\vec d')$, and $\vec d'$
is chosen so that \propref{ross-realizations} applies.
\end{proof}
\begin{lemma}\lemlab{laman-2}
In any realization of $(\tilde{G},\varphi,\vec d^{\perp})$, the Ross-circuits are realized with all their
vertices coincident and on the $y$-axis.
\end{lemma}
\begin{proof}
This follows from how we chose $\vec d$ and \propref{ross-circuit-pairs}.
\end{proof}
As a consequence of \lemref{laman-2}, and the fact that we picked $\vec d$ so that $\vec d^\perp$
extends to a generic assignment of directions $(\vec d^*)^\perp$ on the reduced graph $(G^*,\bgamma)$
we have:
\begin{lemma}
The realization space of $(\tilde{G},\varphi,\vec d^\perp)$ is identified with that of
$(\tilde{G^*},\varphi,\vec (d^*)^\perp)$ which, furthermore, contains only collapsed solutions.
\end{lemma}
Observe that a direction network for a single self-loop (colored $1$) with a generic direction only has solutions where
vertices are collapsed and on the $y$-axis.  Consequently, replacing a Ross-circuit with a single vertex
and a self-loop yields isomorphic realization spaces.  Since the reduced graph is reflection-$(2,2)$ by \propref{reduced-graph} and the directions
assigned to its edges were chosen generically for \propref{reflection-22-collapse},
that $(\tilde{G},\varphi,\vec d^\perp)$ has only collapsed solutions follows.
Thus, we have exhibited a special pair, completing the proof.
\eop

\paragraph{Remark} It can be seen that the realization space of a direction network as supplied by
\theoref{direction-network} has at least one degree of freedom for each edge that is not in a Ross basis.  Thus,
the statement cannot be improved to, e.g., a unique realization up to translation and scale.

\section{Infinitesimal rigidity of reflection frameworks} \seclab{reflection-laman-proof}
Let $(\tilde{G},\varphi,\bm{\ell})$ be a reflection framework and let $(G,\bgamma)$
be the quotient graph.  The configuration space, which is the set of
solutions to the quadratic system \eqref{lengths-1}--\eqref{lengths-2} is canonically
identified with the solutions to:
\begin{eqnarray}\eqlab{colored-lengths}
||\Phi(\gamma_{ij})\cdot \vec p_j - \vec p_i||^2 = \ell^2_{ij} & \qquad
\text{for all edges $ij\in E(G)$}
\end{eqnarray}
where $\Phi$ acts on the plane by reflection through the $y$-axis.  (That ``pinning down''
$\Phi$ does not affect the theory is straightforward from the definition of the configuration
space: it simply removes rotation and translation in the $x$-direction from the
set of trivial motions.)

\subsection{Infinitesimal rigidity}
Computing the formal differential of \eqref{colored-lengths}, we obtain the system
\begin{eqnarray}\eqlab{colored-inf}
\iprod{\Phi(\gamma_{ij})\cdot \vec p_j - \vec p_i}{\vec v_j - \vec v_i} = 0 & \qquad
\text{for all edges $ij\in E(G)$}
\end{eqnarray}
where the unknowns are the \emph{velocity vectors} $\vec v_i$.  A standard kind of result (cf. \cite{AR78})
is the following.
\begin{prop}\proplab{ar-direction}
Let $\tilde{G}(\vec p,\Phi)$ be a realization of an abstract framework $(\tilde{G},\varphi,\bm{\ell})$.
If the corank of the system \eqref{colored-inf} is one, then $\tilde{G}(\vec p)$ is rigid.
\end{prop}
Thus, we define a realization to be \emph{infinitesimally rigid} if the system \eqref{colored-inf}
has maximal rank, and \emph{minimally infinitesimally rigid} if it is infinitesimally rigid but ceases to
be so after removing any edge from the colored quotient graph.

By definition, infinitesimal rigidity is defined by a polynomial condition in the coordinates of
the points $\vec p_i$, so it is a generic property associated with the colored graph $(G,\bgamma)$.

\subsection{Relation to direction networks}
Here is the core of the direction network method for reflection frameworks:
we can understand the rank of \eqref{colored-inf} in terms of a direction network.
\begin{prop}\proplab{rigidity-vs-directions}
Let $\tilde{G}(\vec p,\Phi)$ be a realization of a reflection framework with $\Phi$
acting by reflection through the $y$-axis.  Define the direction $\vec d_{ij}$ to be
$\vec \Phi(\gamma_{ij})\cdot \vec p_j - \vec p_i$.  Then the rank of \eqref{colored-inf}
is equal to that of \eqref{colored-system} for the direction network
$(G,\bgamma,\vec d^{\perp})$.
\end{prop}
\begin{proof}
Exchange the roles of $\vec v_i$ and $\vec p_i$ in \eqref{colored-inf}.
\end{proof}

\subsection{Proof of \theoref{reflection-laman}}
The, more difficult, ``Laman direction'' of the Main Theorem follows immediately from
\theoref{direction-network} and \propref{rigidity-vs-directions}:
given a reflection-Laman graph \theoref{direction-network} produces a realization with
no coincident endpoints and a certificate that \eqref{colored-inf} has corank one.
\eop

\subsection{Remarks}
The statement of \propref{rigidity-vs-directions} is \emph{exactly the same} as the analogous statement for
orientation-preserving cases of this theory.  What is different is that, for reflection frameworks,
the rank of $(G,\bgamma,\vec d^{\perp})$ is \emph{not}, the same
as that of $(G,\bgamma,\vec d)$.  By \propref{reflection-22-collapse}, the set of directions arising as
the difference vectors from point sets are \emph{always non-generic} on reflection-Laman graphs, so
we are forced to introduce the notion of a special pair as in \secref{direction-network}.

\bibliographystyle{plainnat}

\end{document}